\documentclass[10pt,a4paper,preprint,review]{article}

\usepackage{amssymb}
\usepackage{amsmath}
\usepackage{amsthm}
\usepackage[margin=2.5cm]{geometry}

\def\N{\mathbb{N}}

\def\R{\mathbb{R}}

\def\H{\mathbb{H}}

\def\a{\alpha}

\def\k{\kappa}

\def\g{\mathfrak{g}}

\def\p{\mathfrak{p}}

\def\gl{\mathfrak{gl}}

\def\sp{\mathfrak{sp}}

\def\diff{\partial}
\def\codiff{\partial^*}

\def\Abdle{\mathcal{A}}
\def\Dbdle{\mathcal{D}}

\def\Gbdle{\mathcal{G}}

\def\Qbdle{\mathbb{Q}}

\def\Vbdle{\mathcal{V}}

\def\Ric{\mathrm{Ric}}
\def\scal{\mathrm{scal}}

\def\:{\lrcorner}
\def\#{\sharp}

\def\tens{\otimes}

\def\dsum{\oplus}

\def\alt{\wedge}

\def\prod{\times}

\def\isom{\cong}

\theoremstyle{plain}
\newtheorem{Theorem}{Theorem}[section]
\newtheorem{Lemma}[Theorem]{Lemma}
\newtheorem{Corollary}[Theorem]{Corollary}
\newtheorem{Proposition}[Theorem]{Proposition}

\theoremstyle{definition}

\def\no{\noindent}

{\begin{list}{}%
         {\setlength{\leftmargin}{#1}}%
         \item[]%
}
{\end{list}}

\begin{document}


\title{Weyl connections and the local sphere theorem for quaternionic contact structures}

\author{Jesse Alt}

\maketitle


\begin{abstract}
We apply the theory of Weyl structures for parabolic geometries developed by A. \v{C}ap and J. Slov\'ak in \cite{CS03} to compute, for a quaternionic contact (qc) structure, the Weyl connection associated to a choice of scale, i.e. to a choice of Carnot-Carath\'eodory metric in the conformal class. The result of this computation has applications to the study of the conformal Fefferman space of a qc manifold, cf. \cite{alt2}. In addition to this application, we are also able to easily compute a tensorial formula for the qc analog of the Weyl curvature tensor in conformal geometry and the Chern-Moser tensor in CR geometry. This tensor agrees with the formula derived via independent methods by S. Ivanov and D. Vasillev in \cite{IVlocalsphere}. However, as a result of our derivation of this tensor, its fundamental properties -- conformal covariance, and that its vanishing is a sharp obstruction to local flatness of the qc structure -- follow as easy corollaries from the general parabolic theory.
\end{abstract}



\section{Introduction}

The quaternionic contact (qc) structures introduced by O. Biquard (cf. \cite{B00}, \cite{IMV07}, and definitions in section 2.2 below) are a type of parabolic geometry. In particular, this means that any qc structure $(M,\Dbdle,\Qbdle,[g])$ admits a regular, normal Cartan geometry of a certain type $(G,P)$ -- i.e. a $P$-principal fiber bundle over $M$ together with a Cartan connection on the total space taking values in the Lie algebra $\g$ of $G$ --, which induces the structure and is canonical in the sense that it is unique up to isomorphism. The advantage of this viewpoint is that a qc structure is identified with its canonical Cartan geometry, and the general theory of parabolic geometries then offers many nice applications deriving from the representation theory of semisimple groups and their parabolic subgroups. For example, the existence of curved Bernstein-Gelfand-Gelfand sequences allows one to algorithmically construct many invariant differential operators (cf. \cite{CSS01}, \cite{CD01}; these are applied to the case of qc structures in \cite{CS}). Some other well-known types of parabolic geometries are conformal semi-Riemannian structures and non-degenerate, partially-integrable CR structures of codimension one.\\

The canonical Cartan geometry allows us to determine whether two qc structures are equivalent. In particular, we know whether a qc structure $(M,\Dbdle,\Qbdle,[g])$ is locally isomorphic to the (flat) homogeneous model. (The compact homogeneous model is the sphere $S^{4n+3} \isom G/P$, considered as the ``boundary at infinity'' of quaternionic projective $(n+1)$-space, while the non-compact homogeneous model is the quaternionic Heisenberg group $G(\H) = \H^n \times \mathrm{Im}(\H)$ endowed with its natural qc structure -- $G(\H)$ is isomorphic as a qc manifold to $S^{4n+3}$ minus a point, via the quaternionic Cayley transform.) By fundamental properties of Cartan geometries (cf. e.g. chapter 5 of \cite{Sha}), $(M,\Dbdle,\Qbdle,[g])$ is locally flat if and only if its Cartan geometry is \emph{flat}, i.e. if and only if the curvature of this Cartan geometry vanishes identically. As our canonical Cartan geometry is of parabolic type, this answer can be improved because we know that the curvature vanishes whenever certain of its components vanish. In fact, in the case of qc structures some computable algebraic facts tell us that a single component determines the vanishing of the entire curvature, and this component corresponds to a conformally covariant tensor $W^{qc(2)} \in \Gamma(\Lambda^2(\Dbdle^*)\tens\Lambda^2(\Dbdle^*))$ with curvature-type symmetries. In this respect, it is analogous to the Weyl curvature tensor in conformal geometry and the Chern-Moser tensor \cite{Chern-Moser} in CR geometry.\\

The above criteria for local flatness have the drawback of being abstract. Given a representative Carnot-Carath\'eodory metric $g \in [g]$ for a qc structure $(M,\Dbdle,\Qbdle,[g])$, we would like to have a formula for representing $W^{qc(2)}$ in terms of $g$ and associated geometric tensors, just like we have such formulae for the Weyl tensor and the Chern-Moser tensor in terms, respectively, of a fixed metric representing the conformal class or a pseudo-hermitian form for the CR structure. One aim of the present text is to derive such a geometric-tensorial formula for the qc case, and the result is given by the formula for $W^{qc(2)}$ in theorem \ref{Wqc curvature tensor}.\\

The formula derived here is equivalent to the one defining the tensor $WR$ in \cite{IVlocalsphere}, which is not a surprise given that the tensor $WR$ is shown in that text (by different means) to be conformally covariant and an obstruction to local flatness. However, our approach to discovering this tensor has the advantage of relating it to a component of the canonical Cartan curvature. In particular, once we have identified it as the representation of this component with respect to a fixed Carnot-Carath\'eodory metric $g$, the properties that it is conformally covariant and gives a sharp obstruction to local flatness both follow as easy corollaries from the general theory of parabolic geometries. By contrast, in \cite{IVlocalsphere} the proof e.g. of conformal covariance begins by referencing a transformation formula under conformal rescaling for the first component of the tensor $WR$, which is derived using a computer program and takes up ten pages!\\

The derivation leading to our theorem \ref{Wqc curvature tensor} is basically an application of the theory of Weyl structures, developed for general parabolic geometries in \cite{CS03}. Indeed, as we explain in section 2.2, a choice of Carnot-Carath\'eodory metric $g \in [g]$ is a global section of a bundle of scales, and thus (see section 2.1) uniquely determines an exact Weyl structure for the canonical Cartan geometry of a qc structure $(M,\Dbdle,\Qbdle,[g])$. The computations in sections 3 and 4 amount to inductively determining components of the pullback of the canonical Cartan connection by via this Weyl structure. In particular, theorem \ref{QC Weyl connection} gives an expression for the Weyl connection with respect to $g$. This result is applied in \cite{alt2} to the characterization of the conformal Fefferman space of a qc structure. In section 2.1 we begin with a review of the basic definitions and properties of parabolic geometries and Weyl structures, with emphasis on the facts needed for our computation. This review is in no sense intended to be comprehensive, as a number of useful and easily accessible overviews exist. In particular, for fuller explanation the reader is referred to section 2 of \cite{CS03}. Section 2.2 of the present text then gives the set-up for the main computations (in sections 3 and 4) by explicitly identifying the generally defined objects of section 2.1 in the case of qc structures.\\

\no \textbf{Acknowledgements:} I am grateful to Andreas \v{C}ap for giving me a primer in the methods of \cite{CS03} for computing Weyl structures. Conversations with Ivan Minchev were very helpful for clarifying some basic concepts about qc structures and quaternionic algebra. The main computations for the results presented here were done at Humboldt University in Berlin, while I was supported by the Schwerpunktprogramm 1154 (Global Differential Geometry) of the Deutsche Forschungsgemeinschaft. The text was completed at the University of the Witwatersrand in Johannesburg, while I was supported by a university postdoctoral research fellowship.\\

\section{Parabolic geometries and Weyl structures}

\subsection{General theory}

A \emph{parabolic geometry} is a Cartan geometry $(\Gbdle,\omega)$ of parabolic type $(G,P)$. That is, we have a $P$-principal fiber bundle $\pi: \Gbdle \rightarrow M$ and a Cartan connection $\omega \in \Omega^1(\Gbdle,\g)$ (by definition, this means $\omega$ restricts pointwise to linear isomorphisms of the tangent space $T_u\Gbdle \isom \g$, is equivariant with respect to the actions of $P$ on $\Gbdle$ and $\g$, and respects the fundamental vector fields on $\Gbdle$ determined by $\p$). Here, $\g$ is the Lie algebra of a (real or complex) semisimple Lie group $G$, and $P \subset G$ is a (closed) parabolic subgroup in the sense of representation theory.\\

In particular, the Lie algebra $\g$ admits a $\vert k \vert$-grading for some $k \in \N$: $\g = \g_{-k} \dsum \ldots \dsum \g_k$ as a vector space, and $[\g_i,\g_j] \subset \g_{i+j}$ for all $-k \leq i,j \leq k$ (where $\g_l := \{0\}$ for $\vert l \vert > k$). Moreover, the Lie algebra of $P$ is the subalgebra $\p := \g_0 \dsum \ldots \dsum \g_k$ and $P$ splits via the Levi decomposition as $P = G_0 \ltimes P_+$, with $G_0$ a reductive subgroup having Lie algebra $\g_0$ and $P_+$ a nilpotent group, normal in $P$, with $P_+ = \mathrm{exp}(\g_1 \dsum \ldots \dsum \g_k)$. We make use of the following standard notation: $\g^i := \g_i \dsum \g_{i+1} \dsum \ldots \dsum \g_k$, $\p_+ := \g^1$ and $\g_- := \g_{-k} \dsum \ldots \dsum \g_{-1}$. Note that $\g_{-}$ is isomorphic to the $P$-module $\g/\p$ as a vector space, which endows $\g_{-}$ with the isomorphic $P$-module structure.\\

Given a parabolic geometry $(\Gbdle,\omega)$ of type $(G,P)$, the filtration $\g^{-k} \supset \ldots \supset \g^{-1}$ induces a filtration $T^{-k}M \supset \ldots \supset T^{-1}M$ of the tangent bundle $TM$, via the isomorphism $TM \isom \Gbdle \times_{\mathrm{Ad}(P)} (\g/\p)$ which holds for any Cartan geometry. Alternatively, for the associated \emph{adjoint bundle} $\Abdle := \Gbdle \times_{\mathrm{Ad}(P)} \g$, the ($\mathrm{Ad}(P)$-invariant) filtration of $\g$ determines a filtration $\Abdle^{-k} \supset \ldots \supset \Abdle^k$ of $\Abdle$, and the isomorphism $TM \isom \Abdle/\Abdle^0$ defines the filtration of $TM$. In both cases, we also have the \emph{associated graded bundles}
\begin{align*}
\mathrm{Gr}(\Abdle) = \Abdle_{-k} \dsum \ldots \dsum \Abdle_{k} \,\, \mathrm{and} \,\,
\mathrm{Gr}(TM) = \mathrm{Gr}_{-k}(TM) \dsum \ldots \dsum \mathrm{Gr}_{-1}(TM),
\end{align*}
\no where $\Abdle_{i} := \Abdle^{i}/\Abdle^{i+1}$ and $\mathrm{Gr}_i(TM) := T^{i}M/T^{i+1}M$.\\

Consider the curvature form $\Omega^{\omega} \in \Omega^2(\Gbdle,\g)$ of $\omega$ defined by the structure equation $$\Omega^{\omega}(u,v) = d\omega(u,v) + [\omega(u),\omega(v)].$$ This can be identified via standard arguments ($\Omega^{\omega}$ is $P$-equivariant and horizontal) either to a $2$-form on $M$ with values in the adjoint bundle  -- $K \in \Omega^2(M,\Abdle)$ -- or to a $P$-equivariant function $\k \in C^{\infty}(\Gbdle,\Lambda^2(\g_{-})^* \tens \g)$. Identification with the latter object determines a well-defined notion of homogeneity components $\k^{(2-k)}, \ldots, \k^{(3k)}$, defined by: $\k^{(l)}(\g_{i},\g_{j}) \subset \g_{l+i+j}.$\\

It is a fact (cf. discussion and reference in section 2.7 of \cite{CS03}) that the curvature $\k$ has strictly positive homogeneity (i.e. $\k^{(l)} = 0$ for all $l \leq 0$, in which case we say $(\Gbdle,\omega)$ is \emph{regular}) if and only if the associated graded tangent bundle $\mathrm{Gr}(TM)$ (with the bracket naturally induced by the Lie bracket of vector fields) is pointwise isomorphic to $\g_-$ as a Lie algebra. Furthermore, we have a $G_0$-principal bundle $\pi_0:\Gbdle_0 \rightarrow M$ defined by $\Gbdle_0 := \Gbdle/P_+$ (this also defines a $P_+$-principal bundle $\pi_+:\Gbdle \rightarrow \Gbdle_0$), and $\omega$ descends to a partial frame form on $\Gbdle_0$, giving a reduction of the frame bundle of $\mathrm{Gr}(TM)$ to the group $G_0$. For a manifold $M$ of the same dimension as $\g_-$, a filtration of $TM$ such that $\mathrm{Gr}(TM)$ is pointwise isomorphic to $\g_-$, along with a reduction of the structure group of $\mathrm{Gr}(TM)$ to a group $G_0$ having Lie algebra $\g_0$, is called a \emph{regular infinitesimal flag structure} of type $(\g,\p)$, and we'll denote such a structure by $(\Gbdle_0,\{T^iM\})$.\\

A regular parabolic geometry $(\Gbdle,\omega)$ of type $(G,P)$ determines a regular infinitesimal flag structure of type $(\g,\p)$ on $M$, and the fundamental theorem of parabolic geometry gives a converse as follows. The curvature function $\k$ on $\Gbdle$ takes values in $\Lambda^2(\g_-)^* \tens \g$, and this space can be identified with $C^2(\g_-,\g)$, the bilinear alternating functions on $\g_-$ with values in $\g$, which is part of the complex $C^*(\g_-,\g)$ of co-chains on $\g_-$ with values in $\g$. A co-differential $\codiff: C^{q+1}(\g_-,\g) \rightarrow C^q(\g_-,\g)$ is determined by the following formula (cf. \cite{Yam} p. 468, where it is attributed to the work of N. Tanaka and B. Kostant):
\begin{align}
(\codiff \varphi)(X_1,\ldots,X_q) =& \sum_{i=1}^n [e^i,\varphi(e_i,X_1,\ldots,X_q)] - \frac{1}{2} \sum_{i=1}^{n}\sum_{j=1}^{q}(-1)^{j}\varphi([e^i,X_j]_-,e_i,X_1,\ldots,\hat{X}_j,\ldots,X_q). \label{codiff formula}
\end{align}
\no Here, $\varphi \in C^{q+1}(\g_-,\g)$ and $X_1,\ldots,X_q \in \g_-$, while $\{e_1,\ldots,e_n\}$ denotes a basis of $\g_-$ and $\{e^1,\ldots,e^n\}$ is a dual basis of $\p_+$ with respect to the Killing form. The notation $X_-$ denotes projection onto $\g_-$ of a vector $X \in \g$. The codifferential $\codiff$ is $P$-equivariant, and allows us to define a parabolic geometry as \emph{normal} precisely when $\codiff \circ \k = 0$ holds for the curvature function. The fundamental theorem of parabolic geometry (cf. corollary 3.23 and theorem 4.5 of \cite{CS99}, which gives a slight generalization of the result originally due to \cite{Tan79}) states that for any regular infinitesimal flag structure on a manifold $M$, there exists a regular, normal parabolic geometry $(\Gbdle,\omega^{nc})$ which induces it. The process by which this parabolic geometry is (abstractly) constructed is an iterated ``prolongation''. The result is unique up to isomorphism in almost all cases, while in the exceptional cases (e.g. projective structures) a choice of first prolongation determines the parabolic geometry uniquely.\\

Now, invariants of the underlying infinitesimal flag structure can be determined from its canonical Cartan geometry $(\Gbdle,\omega^{nc})$. In particular, the underlying structure is locally geometrically isomorphic to the homogeneous model $G/P$ if and only if the Cartan curvature $\k^{nc}$ vanishes identically, i.e. when $(\Gbdle,\omega^{nc})$ is flat. For parabolic geometries this can be refined using algebraically determined limits on the possible values of $\k^{nc} \in C^{\infty}(\Gbdle,C^2(\g_-,\g))$. A differential $\diff: C^q(\g_-,\g) \rightarrow C^{q+1}(\g_-,\g)$ is given by:
\begin{align*}
(\diff \varphi)(X_0,\ldots,X_q) := \sum_{i=0}^q (-1)^i[X_i,&\varphi(X_0,\ldots,\hat{X_i},\ldots,X_q)] \\ &+ \sum_{i<j}(-1)^{i+j}\varphi([X_i,X_j]_-,X_0,\ldots,\hat{X_i},\ldots,\hat{X_j},\ldots,X_q).
\end{align*}
\no Then $\diff$ and $\codiff$ are adjoint with respect to some positive-definite inner-product defined on the $C^q(\g_-,\g)$. The resulting Kostant Laplacian $\Box: C^q(\g_-,\g) \rightarrow C^q(\g_-,\g)$, defined by $\Box = \diff \circ \codiff + \codiff \circ \diff$, is self-adjoint, $G_0$-equivariant, preserves homogeneity, and we have $H^p_l(\g_-,\g) \isom \mathrm{Ker}(\Box_{\vert C^p_l(\g_-,\g)})$. These cohomology groups may be computed algorithmically using Kostant's theory (cf. \cite{Yam}). Together with these algorithms, the following fact about the curvature of a normal parabolic geometry (cf. corollary 4.10 of \cite{CS99}) gives useful simplifications of the criteria for flatness:

\begin{Proposition} \label{dell of lowest} Let $\k^{nc} = \sum_{i=1}^{3k} \k^{nc(i)}$ be the splitting into homogeneity components of the curvature of a regular and normal parabolic geometry. If $\k^{nc(j)} \equiv 0$ for all $j < l$, then we have $\diff \circ \k^{nc(l)} \equiv 0$ and hence, $\k^{nc(l)}(u) \in \mathrm{Ker}(\Box) \isom H^2_l(\g_-,\g)$ for all $u \in \Gbdle$.
\end{Proposition}

Weyl structures for general parabolic geometries were defined in \cite{CS03}, generalizing the notion of Weyl structure in conformal geometry: A Weyl structure is a global, $G_0$-equivariant section $\sigma$ of the canonical projection $\pi_+:\Gbdle \rightarrow \Gbdle_0$. By proposition 3.2 of \cite{CS03}, global Weyl structures always exist for parabolic geometries in the real (smooth) category, and they exist locally in the holomorphic category. A choice of Weyl structure $\sigma$ determines a reduction of $\Gbdle$ to the structure group $G_0$, and this may be used to decompose any associated vector bundle into irreducible components with respect to $G_0$. In particular, it determines an isomorphism of the adjoint tractor bundle with its associated graded bundle $\Abdle \isom^{\sigma} \mathrm{Gr}(\Abdle).$\\

Considering the pull-back of the Cartan connection, $\sigma^*\omega$, the $\vert k \vert$-grading of $\g$ gives a decomposition into $G_0$-invariant components, $\sigma^*\omega = \sigma^*\omega_{-k} + \ldots + \sigma^*\omega_k = \sigma^*\omega_- + \sigma^*\omega_0 + \sigma^*\omega_+.$ Since $\sigma$ commutes with fundamental vector fields, it follows from the defining properties of the Cartan connection (cf. 3.3 of \cite{CS03}) that $\sigma^*\omega_i$ is horizontal for all $i \neq 0$, while $\sigma^*\omega_0$ defines a principal $G_0$ connection for $\Gbdle_0 \rightarrow M$. Thus the negative and positive components descend to 1-forms on $M$ with values in the negative and positive components of the graded adjoint bundle, respectively, and we will simply identify them as such. The negative component $\sigma^*\omega_- \in \Omega^1(M;\Abdle_{-k} \dsum \ldots \dsum \Abdle_{-1})$ is called the \emph{soldering form} of $\sigma$, and defines an isomorphism $TM \isom \mathrm{Gr}(TM) \isom \Abdle_-$. The positive component, denoted by $\mathsf{P} := \sigma^*\omega_+ \in \Omega^1(M;\Abdle_1 \dsum \ldots \dsum \Abdle_k)$, is called the \emph{Rho-tensor} and generalizes the Schouten tensor of conformal geometry. The connection $\sigma^*\omega_0 \in \Omega^1(\Gbdle_0,\g_0)$ is called the \emph{Weyl connection}.\\

In the classical theory of Weyl structures for a conformal manifold $(M,[g])$, cf. \cite{Weyl}, a Weyl connection (that is, a torsion-free linear connection $\nabla$ on $M$ preserving the conformal class $[g]$) is actually determined by a relatively small piece of information, namely the connection it induces on the ray bundle $\mathcal{Q} \rightarrow M$ of conformal metrics. An analogous property is established for general parabolic geometries in \cite{CS03}, via the introduction of \emph{scale bundles}, which are distinguished $\R^+$ principal bundles $\mathcal{L}^{\lambda} \rightarrow M$ associated to \emph{scale representations} $\lambda: G_0 \rightarrow \R^+$, defined as follows. A \emph{scaling element} $\varepsilon \in \g_0$ is an element of the center $\mathfrak{z}(\g_0)$ which acts via the adjoint representation by scalar multiplication on each grading component $\g_i$: $[\varepsilon,X] = s_i X$ for all $X \in \g_i$ and $s_i \in \R$. A scale representation associated to $\varepsilon$ is then a representation $\lambda: G_0 \rightarrow \R^+$ such that $\lambda'(A) = B_{\g}(A,\varepsilon)$ holds for the derivative of $\lambda$ and any $A \in \g_0$ (for $B_{\g}$ the Killing form of $\g$), and a scale bundle is any $\R^+$ bundle associated to $\Gbdle_0$ by a scale representation.\\

In section 3 of \cite{CS03}, it is shown that scaling elements always exist, that they induce unique scaling representations (and hence scale bundles), and this allows us to associate a connection form $\omega^{\lambda} \in \Omega^1(\mathcal{L}^{\lambda})$ on a chosen scale bundle to the Weyl connection $\sigma^*\omega_0$ of any Weyl structure $\sigma$ ($\omega^{\lambda}$ is induced by $\lambda' \circ \sigma^*\omega_0 \in \Omega^1(\Gbdle_0)$ under the identification $\mathcal{L}^{\lambda} \isom \Gbdle_0/\mathrm{Ker}(\lambda)$, and we have the corresponding fact for its curvature, cf. lemma 3.8 of \cite{CS03}). In fact, by theorem 3.12 of \cite{CS03}, this gives a bijection between Weyl connections on $\Gbdle_0$ and connection forms on $\mathcal{L}^{\lambda}$, generalizing Weyl's theorem in conformal geometry. This gives a notion of distinguished \emph{closed} and \emph{exact} Weyl structures, whose corresponding connection forms on $\mathcal{L}^{\lambda}$ are flat and globally trivial (i.e. induced by a global section of $\mathcal{L}^{\lambda}$), respectively.\\

It follows that the entire pull-back $\sigma^*\omega \in \Omega^1(\Gbdle_0,\g)$ can in principle be recovered, for an exact Weyl structure, from a global section of $\mathcal{L}^{\lambda}$ which (as we'll see in the next section for qc structures) is generally a relatively simple object. Indeed, section 4 of \cite{CS03} is concerned with characterizing, for a regular, normal parabolic geometry $(\Gbdle,\omega^{nc})$, when a $\g$-valued one-form $\omega \in \Omega^1(\Gbdle_0,\g)$ corresponds to a Weyl structure $\sigma$, i.e. when $\omega = \sigma^*\omega^{nc}$. Using this procedure, it is possible to inductively determine the homogeneity components of $\sigma^*\omega^{nc}$, or equivalently the Weyl structure $\sigma$, from a fixed structure as simple as a global section of $\mathcal{L}^{\lambda}$. We give here the basic notions involved, and the method will be demonstrated in the case of qc structures via the computations in sections 3 and 4 below.\\

A \emph{Weyl form} for a regular infinitesimal flag structure $(\Gbdle_0,\{T^iM\})$, is a $\g$-valued one-form $\omega \in \Omega^1(\Gbdle_0,\g)$ which is $G_0$-equivariant, respects the fundamental vector fields of $\Gbdle_0$, and whose negative components $\omega_- \in \Omega^1(\Gbdle_0,\g_-)$ agree with the partial frame forms on $\Gbdle_0$ defining the reduction of $\mathrm{Gr}(TM)$ to $G_0$. The \emph{Weyl curvature} $W \in \Omega^2(\Gbdle_0,\g)$ of $\omega$ is defined by $W(u,v) = d\omega(u,v) + [\omega(u),\omega(v)]$, and its \emph{total curvature} $K = K_{\leq} + K_+ \in \Omega^2(\Gbdle_0,\g)$ is defined by $K_{\leq}(u,v) = d\omega_{\leq}(u,v) + [\omega_{\leq}(u),\omega_{\leq}(v)]$ and $K_+(u,v) = d\omega_+(u,v) + [\omega_+(u),\omega_+(v)]$, where $\omega_{\leq} = \omega_- + \omega_0$. The following, collecting facts proven in sections 4.2-4.4 of \cite{CS03}, will be essential for our later computations:

\begin{Theorem} \label{Weyl curvature} (\cite{CS03}) Let $(\Gbdle, \omega^{nc})$ be a regular, normal parabolic geometry of type $(G,P)$, $(\Gbdle_0,\{T^iM\})$ the underlying regular infinitesimal flag structure it induces, and $\omega \in \Omega^1(\Gbdle_0,\g)$ a Weyl form. Then $\omega$ is induced by a Weyl structure $\sigma:\Gbdle_0 \rightarrow \Gbdle$ (i.e. $\omega = \sigma^*\omega^{nc}$) if and only if $\codiff W = 0$.\\

\no For an arbitrary Weyl-form $\omega$, we have a natural splitting of $W$ by homogeneity, and $W^{(l)} = 0$ holds for all $l \leq 0$. Identifying both $K$ and $W$ with $\Abdle$-valued two-forms on $M$, and denoting e.g. $K_i$ for the $\Abdle_i$-component, we have the following identity determining $K_i$ for any $i < 0$:
\begin{align}
K_i(u,v) = \nabla_u (\omega_i(v)) - \nabla_v (\omega_i(u)) - \omega_i([u,v]) + \sum_{j,k < 0, j+k=i} \{\omega_j(u),\omega_k(v)\}; \label{total curv neg}
\end{align}
\no and the component $K_0$ is determined by the identities, for any $j = -k,\ldots,k$ and $w \in \Abdle_j$:
\begin{align}
\{K_0(u,v),w\} = R_j(u,v) w. \label{total curv zero}
\end{align}

\no (Here, $\nabla$ denotes the covariant derivative induced on $\Abdle_i$ by $\omega_0$, and $R_j(u,v)$ is the curvature endomorphism on $\Abdle_j$.) Finally, we have: $W(u,v) = K(u,v) + \{\mathsf{P}(u),v\} - \{\mathsf{P}(v),u\}$, where $\mathsf{P} \in \Omega^1(M,\Abdle_+)$ is the tensor corresponding to $\omega_+$. In particular, $W^{(l)} = K^{(l)}$ for all $l \leq 1$.
\end{Theorem}

\subsection{The flag structure and scale bundle of a qc manifold}

A qc manifold $\mathcal{M} = (M,\Dbdle,\mathbb{Q},[g])$ of dimension $4n+3$ (cf. definitions in \cite{B00}, \cite{IMV07}) is given by the following: $\Dbdle \subset TM$ is a distribution of real rank $4n$ and co-rank $3$; $\Qbdle \rightarrow M$ is a $S^2$-bundle of almost-complex structures on $\Dbdle$, admitting local sections $(I_1,I_2,I_3)$ (i.e. $\Qbdle_x = \{(a_1I_1(x),a_2I_2(x),a_3I_3(x)) \, \vert \, \sum_{i=1}^3a_i^2 =1\}$) which satisfy the quaternionic relations $I_1 \circ I_2 = -I_2 \circ I_1 = I_3$; and $[g]$ is a conformal equivalence class of Carnot-Carath\'eodory metrics on the distribution $\Dbdle$. Moreover, we require that $\Dbdle$ is given as the kernel of locally-defined $1$-forms $\{\eta^1,\eta^2,\eta^3\} = \eta \in \Omega^1(M,\R^3)$ which satisfy the relations
\begin{align}
d\eta^a(u,v) = 2g(I_au,v) \label{qc relation}
\end{align}
\no for any $u,v \in \Dbdle$ and some (uniquely determined) $(I_1,I_2,I_3) \in \Qbdle$, $g \in [g]$. When $n=1$ an additional integrability condition is assumed, which is that the local $1$-forms $\eta$ may be chosen so that $\{d\eta^a_{\vert \Dbdle}\}_{a=1}^3$ form a local, oriented orthonormal basis of $\Lambda^2_+\Dbdle^*$, and local vector fields $\xi_1,\xi_2,\xi_3$ exist which satisfy
\begin{align}
\xi_a \: \eta^b = \delta_a^b \,\, \mathrm{and} \,\, (\xi_a \: d\eta^b)_{\vert \Dbdle} = -(\xi_b \: d\eta^a)_{\vert \Dbdle} \label{integrability}
\end{align}
\no for $a,b=1,2,3$ (for $n > 1$ these vector fields, called \emph{Reeb vector fields}, automatically exist).\\

Now from this standard definition, we determine the regular infinitesimal flag structure of our qc manifold. A depth $2$ foliation of the tangent bundle is given by $T^{-2}M := TM \supset T^{-1}M := \Dbdle$. By the relation (\ref{qc relation}), we see that the associated graded tangent bundle $\mathrm{Gr}(TM)$, endowed with the algebraic bracket induced by the Lie bracket of vector fields, is pointwise isomorphic to the Lie sub-algebra $\g_- \subset \g \isom \sp(n+1,1)$ determined by the $\vert 2 \vert$-grading of $\g$ given by diagonal components in the following matrix representation:
\begin{align}
\g = \{ \left(\begin{array}{ccc}
            a & z & q \\
            \overline{x} & A_0 & -\overline{z}^t \\
            \overline{p} & -x^t & -\overline{a}
                \end{array}\right)  \, \vert \, a \in \H, A_0 \in \sp(n), p,q \in \mathrm{Im}(\H), x,z^t \in \H^n \}. \label{matrix form}
\end{align}
\no Note that we have isomorphisms $\g_{-1} \isom \H^n, \g_0 \isom \R \dsum \sp(1) \dsum \sp(n)$, etc. from this representation. We will also use these identifications for economy of notation, identifying e.g. the matrix in the above form with only $x \neq 0$ with the vector $\overline{x} \in \H^n$. Let $\tilde{G} \isom Sp(n+1,1)$ denote the obvious matrix group with Lie algebra $\g$, and $G := \tilde{G}/\{\pm \mathrm{Id}\}$. In an obvious way, we will also denote by $\tilde{G}_0 \subset \tilde{G}$ ($G_0 \subset G$) and $\tilde{P} \subset \tilde{G}$, etc. the subgroups with Lie algebras $\g_0$ and $\p$. In particular:
\begin{align*}
\tilde{G}_0 = \{ \left(\begin{array}{ccc}
            sz & 0 & 0 \\
            0 & A & 0 \\
            0 & 0 & s^{-1}\overline{z}
                \end{array}\right)  \, \vert \, s \in \R^+, z, z^{-1}=\overline{z} \in Sp(1), A \in Sp(n) \} \isom \R^+ \times Sp(1) \times Sp(n),
\end{align*}
\no and $\tilde{P} = \tilde{G}_0 \ltimes \mathrm{exp}(\p_+)$.\\

Now, a reduction of the structure group of $\mathrm{Gr}(TM)$ to $G_0 \isom CSp(1)Sp(n)$ is given by the principal bundle $\pi_0: \Gbdle_0 \rightarrow M$, with fibers $$(\Gbdle_0)_x := \{(e_1,\ldots,e_{4n}) \,\, \mathrm{symplectic} \, \mathrm{bases} \, \mathrm{of} \, \Dbdle_x \, \mathrm{w.r.t.} \, g_x, \, (I_1,I_2,I_3) \, \vert \, g \in [g], \, (I_1,I_2,I_3) \in \Qbdle_x\}.$$ We specify the $G_0$-action on $\Gbdle_0$ simultaneously with the partial frame-forms $\omega_{-1} \in \Gamma((T^{-1}\Gbdle_0)^* \tens \g_{-1})$, $\omega_{-2} \in \Omega^1(\Gbdle_0,\g_{-2})$, which identify $\Gbdle_0$ as a reduction of the structure group of $\mathrm{Gr}(TM)$ to $G_0$. Given $u = (e_1,\ldots,e_{4n}) \in (\Gbdle_0)_x$ and $\tilde{\xi} \in T_u^{-1}\Gbdle_0 := (T_u \pi_0)^{-1}(\Dbdle_{x})$, we have $\xi := T_u\pi_0(\tilde{\xi}) = \sum_{a=1}^{4n} \xi^a e_a \in \Dbdle_{x}$. Hence, $u$ determines a bijection $[u]:\Dbdle_{x} \rightarrow \H^n$ by $$[u](\xi) = \overline{x} := \sum_{\alpha =1}^n \overline{x_{\alpha}}d_{\alpha} =: \omega_{-1}(\tilde{\xi}) \in \H^n \isom \g_{-1},$$ where $x_{\alpha} := \xi^{4\alpha-3} + i\xi^{4\alpha-2} + j\xi^{4\alpha-1} + k\xi^{4\alpha} \in \H$ and $\{d_1,\ldots,d_n\}$ is the standard basis (over $\H$) of $\H^n$. Now we determine the action of $G_0$ on bases of $\Dbdle_x$ via a preferred representation on $\H^n \isom \g_{-1}$. Starting with $(s,z,A) \in \tilde{G}_0$, let $R_{(s,z,A)}$ act on $\Dbdle_x$ as determined by the following rules: $[u](R_s(e_a)) = s^{-1}[u](e_a)$ (i.e. $R_s(e_a) = s^{-1}e_a$); $[u](R_{\overline{z}}(\xi)) = ([u](\xi))\overline{z}$ for any $\xi \in \Dbdle_{x}$; and $[u](R_A(\xi)) = A([u](\xi))$ for any $\xi \in \Dbdle_{x}$. If we define $R_{(s,z,A)}(u) := (R_{(s,z,A)}(e_1),\ldots,R_{(s,z,A)}(e_{4n}))$, then we see that this preserves fibers $(\Gbdle_0)_x$ since the latter basis of $\Dbdle_x$ is symplectic with respect to $\tilde{g}_x := s^2 g_x$ and $(\tilde{I_1},\tilde{I_2},\tilde{I_3}) := (\mathrm{Ad}_z(I_1),\mathrm{Ad}_z(I_2),\mathrm{Ad}_z(I_3))$, where we identify $\Qbdle_x \isom S^2 = \{q \in \mathrm{Im}(\H) \vert q\overline{q} = 1\}$ and let $Ad_z$ act on the imaginary quaternions of unit length, for $z \in Sp(1) = \{z \in \H \vert z\overline{z} =1\}$, by $Ad_z:q \mapsto zq\overline{z}$. One verifies that the kernel of this action on each fiber is $\{\pm \mathrm{Id} \in \tilde{G}_0\}$, and so $\Gbdle_0$ is a $G_0$ principal bundle. Furthermore, by construction of the $G_0$-action, $\omega_{-1}$ is equivariant with respect to the $G_0$-module $(\H^n,\rho_{-1}) \isom (\g_{-1},\mathrm{Ad}(G_0))$ given by $\rho_{-1}([(s,z,A)])(\overline{x}) = s^{-1}A(\overline{x})\overline{z}$. In particular, this identifies as an associated bundle $\Dbdle \isom \Gbdle_0 \times_{\rho_{-1}} \H^n \isom \Gbdle_0 \times_{\mathrm{Ad}(G_0)} \g_{-1}$. Now we can also define $\omega_{-2} \in \Omega^1(\Gbdle_0,\g_{-2})$ by letting $\omega_{-2}(\tilde{\xi}) := \overline{p} \in \mathrm{Im}(\H) \isom \g_{-2}$ (where $p := \eta^1(\xi)i + \eta^2(\xi)j + \eta^3(\xi)k$ for $\eta = \{\eta^1,\eta^2,\eta^3\}$ any local qc contact form corresponding to $g_x$ and $(I_1,I_2,I_3)$), and translating via $G_0$-equivariance with respect to the representation $\rho_{-2}([s,z,A]): \bar{p} \mapsto s^{-2} z \bar{p} \bar{z}$.\\

We now specify the natural scale bundle of a qc manifold. An obvious scaling element to take is the grading element $\varepsilon_0 = (1,0,0) \in \R \dsum \sp(1) \dsum \sp(n) \isom \g_0$ as in (\ref{matrix form}), i.e.: $$\varepsilon_0 := \left(\begin{array}{ccc} 1 & 0 & 0 \\0 & 0 & 0 \\0 & 0 & -1 \end{array}\right).$$ Clearly, we have $[\varepsilon_0,X] = jX$ for any $X \in \g_j, j=-2,\ldots,2$. Instead of the Killing form on $\g$, we set our conventions by letting $B$ be one-half the real trace form, i.e. $B(X,Y) := 1/2 \mathrm{Re}(\mathrm{tr}(X \circ Y))$ for $X,Y \in \g$. Using this, the requirement $\lambda'(A) = B(A,\varepsilon_0)$, leads us to conclude that $\lambda: [(s,z,A)] \mapsto s$ is the scale representation $\lambda: G_0 \rightarrow \R^+$ corresponding to $\epsilon_0$. Evidently, $$\mathrm{Ker}(\lambda) = \{ \left(\begin{array}{ccc}
            z & 0 & 0 \\
            0 & A & 0 \\
            0 & 0 & \overline{z}
\end{array}\right)/\{\pm \mathrm{Id}\}  \, \vert \, z \in Sp(1), A \in Sp(n) \} \isom Sp(1)Sp(n).$$ We see also that $\mathcal{L}^{\lambda} \isom \Gbdle_0/\mathrm{Ker}(\lambda) \isom \mathcal{Q}$, where $\mathcal{Q} \rightarrow M$ is the bundle of conformal Carnot-Carath\'eodory metrics: $\mathcal{Q}_x := \{g_x \in S^2(\Dbdle_x^*) \, \vert \, g \in [g]\}$, with the $\R^+$-action given by $s.g_x = s^2g_x$.\\

The above shows in particular that any choice $g \in [g]$ of representative Carnot-Carath\'eodory metric in the conformal class, gives us a global section of $\mathcal{L}^{\lambda}$ and hence corresponds to a unique exact Weyl structure $\sigma_g$ for a parabolic geometry $(\Gbdle,\omega)$ inducing $\mathcal{M}$. We always have existence of the canonical parabolic geometry $(\Gbdle,\omega^{nc})$ of type $(G,P)$, which is characterized up to isomorphism by the fact that its curvature $\k^{nc} \in C^{\infty}(\Gbdle,C^2(\g_-,\g))$ has strictly positive homogeneity and satisfies $\codiff \circ \k^{nc} \equiv 0$. In fact, calculations using Kostant's generalization of the Bott-Borel-Weil theorem show that for $n > 1$ the only non-zero component of $H^2(\g_-,\g)$ occurs in homogeneity two and is a submodule of $\Lambda^2(\g_{-1})^* \tens \g_0$. Thus applying proposition \ref{dell of lowest}, we see that $\k^{nc(2)} \in C^{\infty}(\Gbdle,\Lambda^2(\g_{-1})^* \tens \g_0)$ is the first curvature component which might not vanish (in particular, it follows that $\omega^{nc}$ is \emph{torsion-free}, i.e. $\k^{nc}_i \equiv 0$ for all $i < 0$), and this component vanishes if and only if $(\Gbdle,\omega^{nc})$ is flat. This translates via $\sigma_g$ to a tensorial quantity $W^{qc(2)} \in \Gamma(\Lambda^2\Dbdle^* \tens \Abdle_0)$ corresponding to $\sigma_g^*\Omega^{nc(2)}$, which will be computed in the sequel. (For the case $n=1$, $H^2(\g_-,\g)$ has an additional non-vanishing component of homogeneity one. However, the integrability condition (\ref{integrability}) ensures that the curvature component corresponding to this homogeneity automatically vanishes for $n=1$ as well, cf. lemma \ref{total curv homog 1} in the next section, so the above considerations still apply.)\\

Finally, we note for the sake of completeness how, given an exact Weyl structure $\sigma: \Gbdle_0 \rightarrow \Gbdle$, to recover the corresponding Carnot-Carath\'eodory metric: If $\sigma$ is exact then we have a reduction $r: \Gbdle_0^{\mathrm{hol}} \subset \Gbdle_0$ to structure group $\mathrm{Ker}(\lambda)$ corresponding to the reduced holonomy of $\sigma^*\omega_0$. Then for any $x \in M$ and $\xi,\eta \in \Dbdle_x$, choose a $u \in (\Gbdle_0^{\mathrm{hol}})_x$ and $\tilde{\xi},\tilde{\eta} \in T_u\Gbdle_0^{\mathrm{hol}}$ which project to $\xi$ and $\eta$, respectively. Then setting $$g^{\sigma}_x(\xi,\eta) := B(\omega_{-1}(\tilde{\xi}),\overline{\omega_{-1}(\tilde{\eta})}^t),$$ it is not difficult to calculate that $g^{\sigma}_x(\xi,\eta) = \sum_{a=1}^{4n}\xi^a\eta^a$, where $\xi = \sum_{a=1}^{4n}\xi^a e_a$ and $\eta = \sum_{a=1}^{4n}\eta^a e_a$ with respect to $(e_1,\ldots,e_{4n}) = u \in \Gbdle_0^{\mathrm{hol}}$. Hence $g_x^{\sigma}$ it is independent of the choice of point $u$ in the fiber over $x$ and the metric $g^{\sigma}$ thus defined is in the conformal class $[g]$.\\

\section{Computation of the qc Weyl connection}

We carry over the definitions and notation of section 2.2, in particular $\mathcal{M} = (M,\Dbdle,\mathbb{Q},[g])$ is a qc manifold of dimension $4n+3$, assumed integrable in case $n=1$. We denote by $(\Gbdle,\omega^{nc})$ the regular, normal parabolic geometry of type $(G,P)$ inducing $\mathcal{M}$, which we know exists and is unique up to isomorphism, from the general theory. Let $g \in [g]$ be fixed. As shown in section 2.2, $g$ determines a global section of the scale bundle $\mathcal{L}^{\lambda}$ and hence by theorem 3.12 of \cite{CS03} an exact Weyl structure $\sigma_g: \Gbdle_0 \rightarrow \Gbdle$. Our aim in this section is to compute the non-positive (and especially the degree zero) components of the pull-back $\sigma_g^*\omega^{nc} \in \Omega^1(\Gbdle_0,\g)$. Using the method of computation outlined for general parabolic geometries in section 4 of \cite{CS03}, this can be done by starting with a Weyl form $\omega \in \Omega^1(\Gbdle_0,\g)$ (see section 2.1) and inductively adjusting the homogeneity components to get $\omega^{qc}$ such that the Weyl curvature $W^{qc}$ of $\omega^{qc}$ satisfies $\codiff \circ W^{qc} = 0$. This is equivalent, by theorem \ref{Weyl curvature} cited in section 2.1, to $\sigma_g^*\omega^{nc} = \omega^{qc}$. First we identify a convenient set-up for the subsequent computations by fixing the negative components $\omega^{qc}_- \in \Omega^1(\Gbdle_0,\g_-)$ and identifying the graded adjoint bundle $\mathrm{Gr}(\Abdle) = \Gbdle_0 \times_{\mathrm{Ad}(G_0)} \g$ with a more geometrically familiar vector bundle over $\mathcal{M}$.\\

Note that $g$ uniquely determines a complement $\Vbdle$ to $\Dbdle$: By (\ref{integrability}) we have for any local section $(I_1,I_2,I_3)$ of $\Qbdle$ (which together with $g$ determines a local qc contact form $(\eta^1,\eta^2,\eta^3)$ via (\ref{qc relation})) unique Reeb vector fields $\xi_1,\xi_2,\xi_3$ which locally span a linear complement of $\Dbdle$. A transition $(I_1,I_2,I_3) \mapsto (\tilde{I_1},\tilde{I_2},\tilde{I_3})$ to a different local section of $\Qbdle$ is given by a smooth, locally defined $SO(3)$-valued function $\Phi$ on $M$, and the corresponding local qc contact form $(\tilde{\eta}^1,\tilde{\eta}^2,\tilde{\eta}^3)$ which the new local section together with $g$ determines, satisfies $d\tilde{\eta}^a_{\vert \Dbdle} = \sum_{b=1}^3 \Phi^a_b d\eta^b_{\vert \Dbdle}$. Using this, one calculates that the vector fields $\tilde{\xi}_a := \sum_{b=1}^3 \Phi_b^a \xi_b$ for $a=1,2,3$ satisfy the relations (\ref{integrability}) with respect to $\tilde{\eta}$ and thus determine the new Reeb vector fields. In particular, the local linear complement to $\Dbdle$ which they determine is the same as that determined by $\xi_1,\xi_2,\xi_3$, so $TM \isom^g \Vbdle \dsum \Dbdle \isom \mathrm{Gr}(TM)$ is a global decomposition induced by $g$. In general, for $\zeta \in T_xM$ we'll denote by $\zeta_{\Vbdle} + \zeta_{\Dbdle} \in \Vbdle_x \dsum \Dbdle_x$ the projections onto the sub-bundles.\\

Now we can easily define the negative components $\omega^{qc}_- \in \Omega^1(\Gbdle_0,\g_-)$ of the Weyl form associated to $g$. Let $\tilde{\zeta} \in T\Gbdle_0$ be a tangent vector projecting to some $\zeta \in T_xM$ via $T\pi_0$. Then we have $\zeta = \zeta_{\Vbdle} + \zeta_{\Dbdle}$ and $\tilde{\zeta} = \tilde{\zeta}_{\Vbdle} + \tilde{\zeta}_{\Dbdle}$ where $T\pi_0(\tilde{\zeta}_{\Vbdle}) = \zeta_{\Vbdle}$ and $T\pi_0(\tilde{\zeta}_{\Dbdle}) = \zeta_{\Dbdle}$. We let $\omega^{qc}_{-2}(\tilde{\zeta}) := \omega_{-2}(\tilde{\zeta}) = \omega_{-2}(\tilde{\zeta}_{\Vbdle})$ and $\omega^{qc}_{-1}(\tilde{\zeta}) := \omega_{-1}(\tilde{\zeta}_{\Dbdle})$ as defined in section 2.2.\\

Note that $\omega^{qc}_-$ gives us an identification of $TM \isom \Vbdle \dsum \Dbdle \isom \mathrm{Gr}(TM)$ with the associated vector bundle $\Abdle_- = \Gbdle_0 \times_{\mathrm{Ad}(G_0)} \g_-$. The isomorphism $\Dbdle \isom \Gbdle_0 \times_{\rho_{-1}} \H^n \isom \Gbdle_0 \times_{\mathrm{Ad}(G_0)} \g_{-1}$ was already noted in section 2.2, and $\omega^{qc}_{-2}$ clearly induces an isomorphism $\Vbdle \isom \Gbdle_0 \times_{\rho_{-2}} \mathrm{Im}(\H) \isom \Gbdle_0 \times_{\mathrm{Ad}(G_0)} \g_{-2}$. In fact, now we even can identify the whole graded adjoint bundle $\mathrm{Gr}(\Abdle) \isom \Gbdle_0 \times_{\mathrm{Ad}(G_0)} \g$ with a natural vector bundle over $\mathcal{M}$. This is given by:
\begin{align}
\Abdle_{-2} \dsum \Abdle_{-1} \dsum \Abdle_{0} \dsum \Abdle_{1} \dsum \Abdle_{2} &\isom \Vbdle \dsum \Dbdle \dsum \mathrm{End}_0(\Dbdle) \dsum \Dbdle^* \dsum \Vbdle^*, \label{adjoint bdle id}
\end{align}
\no where $\mathrm{End}_0(\Dbdle)_x := \{A = q_0\mathrm{Id} + \sum_{a=1}^3q_aI_a + A_0 \in \mathrm{End}(\Dbdle_x) \, \vert \, q_i \in \R, \, (I_1,I_2,I_3) \in \Qbdle_x, \, A_0 \in \sp(\Dbdle,g)_x\}$. The details of the identifications $\mathrm{End}_0(\Dbdle) \isom \Gbdle_0 \times_{\mathrm{Ad}(G_0)} \g_0$, $\Dbdle^* \isom \Gbdle_0 \times_{\mathrm{Ad}(G_0)} \g_1$ and $\Vbdle^* \isom \Gbdle_0 \times_{\mathrm{Ad}(G_0)} \g_2$, are given in the appendix. This leads to an algebraic commutator defined on $TM \dsum \mathrm{End}_0(\Dbdle) \dsum T^*M$ induced by the Lie bracket of $\g$ (and therefore respecting the grading), which we denote by $\{,\}$ to distinguish it from the Lie bracket of vector fields. The identities for $\{,\}$ are computed in the appendix and given in the formulae (\ref{CommC})--(\ref{CommJ}).\\

From now on, we use the isomorphism (\ref{adjoint bdle id}) to identify the components of any Weyl form $\omega \in \Omega^1(\Gbdle_0,\g)$ by $\omega_{-2} \simeq \theta_{-2} \in \Gamma(T^*M \tens \Vbdle)$, $\omega_{-1} \simeq \theta_{-1} \in \Gamma(T^*M \tens \Dbdle)$, $\omega_0 \simeq \nabla: \Gamma(TM) \rightarrow \Gamma(T^*M \tens TM)$ a covariant derivative on $TM$, etc. In particular, for the negative components $\omega_{-}^{qc}$ already defined above for our fixed Carnot-Carath\'eodory metric $g$, we have $\omega^{qc}_{-2} \simeq \mathrm{pr}_{\Vbdle}$ and $\omega^{qc}_{-1} \simeq \mathrm{pr}_{\Dbdle}$. The algebraic commutator $\{,\}$ also allows us to carry over the codifferential on $C^*(\g_-,\g)$ in a natural way to a linear operator $\codiff: \Omega^q(M;\Abdle) \rightarrow \Omega^{q-1}(M;\Abdle)$ (just substitute $\{,\}$ for $[,]$ and (dual) bases of $TM$ and $T^*M$ in the formula (\ref{codiff formula})). This is what we'll be computing with in the sequel.\\

The degree zero component $\omega^{qc}_0 \simeq \nabla^{qc}$ will be computed in terms of the Biquard connection $\nabla := \nabla^B$, which was discovered by Biquard in \cite{B00} for $n > 1$ and by Duchemin in \cite{Duch} for $n=1$ under the assumption of integrability:

\begin{Theorem} \label{Biquard connection} (\cite{B00},\cite{Duch}) Let $(M,\Dbdle,\mathbb{Q},[g])$ be a qc manifold (integrable in dimension 7). For any $g \in [g]$, there exists a connection $\nabla$ with torsion $T$, uniquely determined by the following conditions:\\
(i) $\nabla$ preserves the decomposition $TM = \Vbdle \dsum \Dbdle$ and the $Sp(1)Sp(n)$ structure on $\Dbdle$, i.e.: $\nabla g = 0$ and $\nabla \mathbb{Q} \subset \mathbb{Q}$;\\
(ii) For all $u, v \in \Dbdle$, we have $T(u,v) = -[u,v]_{\Vbdle}$;\\
(iii) The connection on $\Vbdle$ is induced by the natural identification of $\Vbdle$ with $\sp(1) := \{\sum_{a=1}^3 q_a I_a\} \subset \mathrm{End}_0(\Dbdle)$;\\
(iv) For $\xi \in \Vbdle$, the endomorphism $T_{\xi}:= T(\xi,.)_{\vert \Dbdle} \in \mathrm{End}(\Dbdle)$ lies in $(\sp(1) \dsum \sp(\Dbdle,g))^{\perp} \subset \mathrm{End}(\Dbdle)$.
\end{Theorem}

We will denote by $\omega_0 = \in \Omega^1(\Gbdle_0,\g_0)$ the connection form inducing $\nabla$ for our fixed Carnot-Carath\'eodory metric $g$. Clearly, the covariant derivative $\nabla^{qc}$ corresponding to the Weyl connection $\omega^{qc}_0 \in \Omega^1(\Gbdle_0,\g_0)$ which we wish to compute, satisfies $\nabla^{qc} = \nabla + \alpha^{qc}$ for some uniquely determined $\alpha^{qc} \in \Omega^1(M;\mathrm{End}_0(\Dbdle))$. In fact, we can restrict our attention to connections of the form $\nabla^{\alpha} = \nabla + \alpha$ with $\alpha \in \Gamma(\Vbdle^* \tens \mathrm{End}_0(\Dbdle))$, thanks to the following:

\begin{Lemma} \label{total curv homog 1} With $\omega_{\leq}^{\alpha} = \omega^{qc}_- + \omega_0^{\alpha}$ given by $\omega_0^{\alpha} \simeq \nabla + \alpha$ for $\alpha \in \Gamma(\Vbdle^* \tens \mathrm{End}_0(\Dbdle))$ as above, we have $K^{\alpha(1)} = W^{\alpha(1)} = 0$. \end{Lemma}

\begin{proof} Evidently, $K^{\alpha(1)} = K_{-2}^{\alpha(1)} + K_{-1}^{\alpha(1)}$ with $K_{-2}^{\alpha(1)} \in \Gamma(\Vbdle^* \wedge \Dbdle^* \tens \Vbdle)$ and $K_{-1}^{\alpha(1)} \in \Gamma(\Lambda^2(\Dbdle^*) \tens \Dbdle)$. Using the formula (\ref{total curv neg}), we see for $\xi \in \Vbdle$, $u \in \Dbdle$:
\begin{align*}
K_{-2}^{\alpha(1)}(\xi,u) &= \nabla^{\alpha}_{\xi} u_{\Vbdle} -\nabla^{\alpha}_u \xi_{\Vbdle} - [\xi,u]_{\Vbdle} \\
&= -\nabla_u \xi - [\xi,u]_{\Vbdle}.
\end{align*}

\no And it is shown in proposition II.1.9 of \cite{B00} that $\nabla_u\xi = [u,\xi]_{\Vbdle}$ (this is a result of property (iii) in theorem \ref{Biquard connection}), so we see that $K^{\a (1)}_{-2} = 0$. Now, for the component $K_{-1}^{\alpha(1)}$ of $K^{\alpha(1)}$, a direct calculation as above gives: $K_{-1}^{\alpha(1)}(u,v) = \nabla_u v - \nabla_v u - [u,v]_{\Dbdle}$, which vanishes by property (ii) of theorem \ref{Biquard connection}. Finally, by theorem \ref{Weyl curvature}, we have $W^{\alpha(1)} = K^{\alpha(1)} = 0$. \end{proof}


The next step -- computation of $K^{\alpha(2)}$ and $\codiff K^{\alpha(2)}$ for a certain class of tensors $\alpha$ --, will at the same time determine the tensor $\alpha^{qc} \in \Gamma(\Vbdle^*\tens\mathrm{End}_0(\Dbdle))$ which we are seeking. The strategy for doing this, based on the discussion in section 4 of \cite{CS03}, is as follows. One notes that the formula in theorem \ref{Weyl curvature}, relating the Weyl curvature $W$ and the total curvature $K$, may be written in homogeneity two as $W^{(2)} = K^{(2)} - \diff \mathsf{P}^{(2)}$, where $\diff$ is the operator on $\Abdle$ induced by the Lie algebra differential $\diff: C^q(\g_-,\g) \rightarrow C^{q+1}(\g_-,\g)$ in the same way as with the codifferential $\codiff$. (Note that in our case, $K^{\alpha(2)} = K^{\alpha(2)}_{\leq}$, which only depends on $\omega^{\alpha}_{\leq}$.) Thus, for a fixed $\alpha \in \Gamma(\Vbdle^* \tens \mathrm{End}_0(\Dbdle))$, we have $\codiff W^{\alpha(2)} = 0$ if and only if $\codiff K^{\alpha (2)} = \codiff \diff \mathsf{P}^{(2)}$, which may be rewritten as:
\begin{align}
\codiff K^{\alpha (2)} = \Box \mathsf{P}^{(2)} - \diff \codiff \mathsf{P}^{(2)}. \label{homog 2 eqn}
\end{align}

By the general theory of parabolic geometries, existence and uniqueness of $\alpha$ and $\mathsf{P}^{(2)} \in \Gamma(\Dbdle^* \tens \Dbdle^*)$ solving (\ref{homog 2 eqn}) is guaranteed, however finding explicit solutions can be reduced to solving simpler equations. To see this, denote $\beta := \codiff \mathsf{P}^{(2)} \in \Gamma(\Vbdle^*)$. Considering the restriction of (\ref{homog 2 eqn}) to $\Vbdle$, we see that it's necessary to solve, for arbitrary $\xi \in \Vbdle$:
\begin{align}
(\codiff K^{\alpha (2)})(\xi) + (\diff \beta)(\xi) = 0. \label{homog 2 eqn simp}
\end{align}
\no (This is because the term $\Box \mathsf{P}^{(2)}(\xi)$ vanishes, since $\Box$ acts by scalar multiplication on irreducible $G_0$-modules, so in particular $\Box \mathsf{P}^{(2)} \in \Gamma(\Dbdle^* \tens \Dbdle^*)$ and must vanish on $\Vbdle$.) In fact, it is \textbf{\emph{sufficient}} to find $\alpha, \beta$ which solve (\ref{homog 2 eqn simp}) to determine a solution to (\ref{homog 2 eqn}): Since $\mathrm{Ker}(\Box_{\vert C^1_2(\g_-,\g)}) \isom H^1_2(\g_-,\g) = 0$, $\Box$ has a well-defined inverse and we may define $\mathsf{P}^{(2)} := \Box^{-1}(\codiff K^{\alpha (2)} + \diff \beta)$. By (\ref{homog 2 eqn simp}), this is a section of $\Dbdle^* \tens \Dbdle^*$. Also, from $\mathrm{Ker}(\Box_{\vert C^1_2(\g_-,\g)}) = 0$, the Hodge decomposition becomes $C^1_2(\g_-,\g) = \mathrm{Im}(\diff) \dsum \mathrm{Im}(\codiff)$. Direct from the definition of $\mathsf{P}^{(2)}$, we have $\Box \mathsf{P}^{(2)} = \codiff K^{\alpha (2)} + \diff \beta$. On the other hand, from the definition of $\Box$, we have $\Box \mathsf{P}^{(2)} = \codiff \diff \mathsf{P}^{(2)} + \diff \codiff \mathsf{P}^{(2)}$, and hence $\diff \beta = \diff \codiff \mathsf{P}^{(2)}$ from the Hodge decomposition, and we have a solution to (\ref{homog 2 eqn}).\\

Next we will turn to the computation of explicit tensors $\alpha, \beta$ solving (\ref{homog 2 eqn simp}), essentially by computing $K^{\alpha(2)}$ and $\codiff K^{\alpha(2)}$ for sufficiently general $\alpha \in \Gamma(\Vbdle^*\tens\mathrm{End}_0(\Dbdle))$. Before beginning with the computations, we recall some definitions of geometric tensors associated to the Carnot-Carath\'eodory metric $g$ and its Biquard connection $\nabla$, which were introduced and studied extensively in \cite{IMV07}.\\

The curvature tensor $R$ of $\nabla$ is defined in the usual way: $R(u,v)w = [\nabla_u,\nabla_v]w - \nabla_{[u,v]}w$ as a $(1,3)$-tensor and $R(u,v,w,z) = g(R(u,v)w,z)$ as a $(0,4)$-tensor. The qc-Ricci-tensor $Ric$ is given by $Ric(u,v) = \sum_{a=1}^{4n}R(e_a,u,v,e_a)$ for any $g$-orthonormal local basis $\{e_a\}$ of $\Dbdle$, while the Ricci-type tensors $\tau_s \in \Gamma(\Dbdle^*\tens\Dbdle^*)$ ($s=1,2,3$) are defined for any choice of local section $(I_1,I_2,I_3)$ of $\Qbdle$ (or equivalently any local qc contact form $(\eta^1,\eta^2,\eta^3)$), by: $4n \tau_s(u,v) = \sum_{a=1}^{4n} g(R(e_a,I_s(e_a))u,v)$. The qc scalar curvature of $g$ is defined by $\mathrm{scal} = \sum_{a=1}^{4n} Ric(e_a,e_a) = \sum_{a,b=1}^{4n}R(e_a,e_b,e_b,e_a)$.\\

Two other important tensors $T^0, U \in \Gamma(\Dbdle^*\tens\Dbdle^*)$ are defined in \cite{IMV07} from the torsion $T$ of the Biquard connection: Denoting by $T_{\xi} = T(\xi,.)_{\vert \Dbdle} \in \Gamma(\mathrm{End}(\Dbdle))$ the torsion endomorphism determined by any $\xi \in \Gamma(\Vbdle)$, Biquard showed that $T_{\xi}$ is totally trace-free, i.e. $\sum_{a=1}^{4n}g(T_{\xi}(e_a),e_a) = \sum_{a=1}^{4n}g(T_{\xi} \circ I_s(e_a),e_a) = 0$ for any $(I_1,I_2,I_3) \in \Qbdle$. Decomposing $T_{\xi} = T^0_{\xi} + b_{\xi}$ into its symmetric and anti-symmetric components with respect to $g$, then $T^0_{\xi}$ is traceless and we have $b_{\xi_s} = I_s \circ U^{\sharp}$ for $U^{\sharp} \in \Gamma(\mathrm{End}(\Dbdle))$ a traceless, symmetric $Sp(1)Sp(n)$-invariant endomorphism field which commutes with all $(I_1,I_2,I_3) \in \Qbdle$ (for $n=1$, $U^{\sharp}=0$ and $T_{\xi} = T^0_{\xi}$). In \cite{IMV07}, $T^0$ and $U$ are then defined by $$T^0(u,v) := \sum_{s=1}^3 g(T^0_{\xi_s} \circ I_s(u),v), \,\, U(u,v) := g(U^{\sharp}(u),v),$$ for any $u,v \in \Dbdle$, and it is shown that these are trace-free, symmetric and $Sp(1)Sp(n)$-invariant sections of $\Dbdle^*\tens\Dbdle^*$. Moreover, $T^0 = (T^0)_{[-1]}$ and $U = (U)_{[3]}$, i.e. they belong to the eigenspaces of the eigenvalues $-1$ and $3$, respectively, of the Casimir operator determined by the identities:
\begin{align}
\sum_{s=1}^3T^0(I_s(u),I_s(v)) = -T^0(u,v), \,\,\,\, \sum_{s=1}^3U(I_s(u),I_s(v) = 3U(u,v). \label{sp decomp}
\end{align}
\no The following theorem, derived in \cite{IMV07}, will prove important in our further computations (the form here is extracted from theorem 2.4 of \cite{IVlocalsphere}):

\begin{Theorem} (\cite{IMV07}): On a qc manifold of dimension $4n+3$, for a fixed Carnot-Carath\'eodory metric $g \in [g]$, we have the following identities (with $U=0$ for $n=1$):
\begin{align}
\Ric(u,v) = (2n+2)&T^0(u,v) + (4n+10)U(u,v) + \frac{\scal}{4n}g(u,v); \label{IMV1}\\
\tau_s(u,v) = \frac{n+2}{2n}(T^0(u,&I_sv)-T^0(I_su,v)) + \frac{\scal}{8n(n+2)}g(u,I_sv); \label{IMV3}\\
T(\xi_r,\xi_s) = &-\frac{\scal}{8n(n+2)}\xi_t - [\xi_r,\xi_s]_{\Dbdle}; \label{IMV4}
\end{align}
\end{Theorem}

Now we return to the computation of the Weyl connection $\omega^{qc}_0$.\\

\begin{Lemma} \label{total curv homog 2} For $\omega^{\alpha}_{\leq}$ as above and the homogeneity two component $K^{\alpha (2)}$ of the total curvature, we have $K^{\alpha(2)} = K^{\alpha (2)}_{\leq} = K^{\alpha(2)}_{-2} + K^{\alpha (2)}_{-1} + K^{\alpha (2)}_0$, with $K^{\alpha(2)}_{-2} \in \Gamma(\Lambda^2(\Vbdle^*) \tens \Vbdle), K^{\alpha(2)}_{-1} \in \Gamma(\Vbdle^* \wedge \Dbdle^* \tens \Dbdle)$ and $K^{\alpha(2)}_0 \in \Gamma(\Lambda^2(\Dbdle^*) \tens \mathrm{End}_0(\Dbdle))$, which satisfy the following identities, for $(r,s,t) \sim (1,2,3)$ and for any $\xi \in \Vbdle$, $u \in \Dbdle$:
\begin{align}
K^{\alpha(2)}_{-2}(\xi_r,\xi_s) = -\frac{\scal}{8n(n+2)} &\xi_t + \{\alpha(\xi_r),\xi_s\} - \{\alpha(\xi_s),\xi_r\}; \label{hom 2 a}\\
K^{\alpha(2)}_{-1}(\xi,u) = &T_{\xi}(u) + \alpha(\xi)(u); \label{hom 2 b}\\
K^{\alpha(2)}_0(u,v) = R(u,v)& + 2\sum_{r=1}^3 g(I_r(u),v)\alpha(\xi_r). \label{hom 2 c}
\end{align}
\end{Lemma}

\begin{proof} A direct application of (\ref{total curv neg}), using the definition of $\nabla^{\alpha} = \nabla + \alpha$, gives $K^{\alpha(2)}_{-2}(\xi_r,\xi_s) = T(\xi_r,\xi_s)_{\Vbdle} + \{\alpha(\xi_r),\xi_s\} - \{\alpha(\xi_s),\xi_r\}$. Applying the identity (\ref{IMV4}) cited above, gives (\ref{hom 2 a}).\\

\no Using (\ref{total curv neg}) again, we see: $K^{\alpha(2)}_{-1}(\xi,u) = \nabla_{\xi} u - [\xi,u]_{\Dbdle} + \alpha(\xi)(u)$. The first two terms on the right hand side add, by definition, to $T_{\xi}(u)$, which proves (\ref{hom 2 b}). And the formula (\ref{total curv zero}) leads to (\ref{hom 2 c}), noting the identity $[u,v]_{\Vbdle} = -2\sum_{r=1}^3 g(I_r(u),v)\xi_r$. \end{proof}

Now, to compute $(\codiff K^{\alpha(2)})(\xi)$, for $\xi \in \Vbdle$ (w.l.o.g. take $\xi = \xi_r$ to be one of the Reeb vector fields), we note first of all the following, using the formula (\ref{codiff formula}) for the codifferential (here we use $\{\xi_1,\xi_2,\xi_3,e_1,\ldots,e_{4n}\}$ and its dual basis of $T^*M$ in the formula, which by the identifications used to compute the algebraic brackets $\{,\}$ corresponds to taking a basis of $\g_-$ and its $B$-dual basis of $\p_+$):

\begin{align*}
(\codiff K^{\alpha (2)})(\xi_r) &= \sum_{a=1}^3 \{K^{\alpha}_{-2}(\xi_r,\xi_a),\eta^a\} + \sum_{a=1}^{4n} \{K^{\alpha}_{-1}(\xi_r,e_a),e^a\} - \frac{1}{2} K^{\alpha}_0(\{\xi_r,e^a\},e_a) \\
&= \sum_{a=1}^3 \{K_{-2}(\xi_r,\xi_a),\eta^a\} + \sum_{a=1}^{4n}\{K_{-1}(\xi_r,e_a),e^a\} - \frac{1}{2}K_0(\{\xi_r,e^a\},e_a) \\
&+ \sum_{a=1}^3 \{\{\alpha(\xi_r),\xi_a\}-\{\alpha(\xi_a),\xi_r\},\eta^a\} + \sum_{a=1}^{4n}\{\alpha(\xi_r)(e_a),e^a\} \\
& \,\,\,\,\,\,\, - \sum_{a=1}^{4n}\sum_{b=1}^3 g(I_b(\{\xi_r,e^a\}),e_a)\alpha(\xi_b) \\
&= (\codiff K^{(2)})(\xi_r) + \mathrm{corr}(\alpha)(\xi_r),
\end{align*}

\no where we denote by $K_i$ the total curvature terms for the choice $\alpha=0$, and define the term $\mathrm{corr}(\alpha)(\xi_r)$ to be all the ``correction terms'' involving $\alpha$ in the expression for $(\codiff K^{\alpha(2)})(\xi_r)$. The following lemma will be of use for computing both of the terms which occur:

\begin{Lemma} \label{codiff trace} For any $A \in \mathrm{End}(\Dbdle)$, the map $A \mapsto \sum_{a=1}^{4n} \{A(e_a),e^a\}$ is given by $$A \mapsto \sum_{s=0}^3 \mathrm{tr}_{I_s}(A) I_s + 8A_{\sp(n)},$$ where $\mathrm{tr}_{I_s}(A) := \sum_{a=1}^{4n}g(A(e_a),I_s(e_a))$ and $A_{\sp(n)}$ denotes the projection onto the $\sp(\Dbdle,g)$-component of $A$. \end{Lemma}

\begin{proof} From the formula for $\{u,\varphi\} \in \mathrm{End}_0(\Dbdle)$, for $u \in \Dbdle$ and $\varphi \in \Dbdle^*$ given in (\ref{CommC}) of the appendix, we get for any $v \in \Dbdle$:
\begin{align}
&\sum_{a=1}^{4n}\{A(e_a),e^a\}(v) = \sum_{a=1}^{4n} (g(A(e_a),e_a) -\sum_{s=1}^3g(I_s(A(e_a)),e_a)I_s(v)) \label{line one} \\
&+ \sum_{a=1}^{4n}((g(e_a,v)A(e_a) - g(A(e_a),v)e_a) - \sum_{s=1}^3(g(e_a,I_s(v))I_s(A(e_a)) - g(A(e_a),I_s(v))I_s(e_a))). \label{line two}
\end{align}

\no Clearly, the right-hand side of (\ref{line one}) corresponds to the terms $\sum_{s=0}^3\mathrm{tr}_{I_s}(A)I_s$ in the formula claimed in the lemma. On the other hand, let us denote by $A = A^{\circ} + A^{\alt}$ the splitting with respect to $g$ of $A$ into symmetric and anti-symmetric components, respectively. Also, we may decompose $A = A_{[-1]} + A_{[3]}$ according to the eigenspaces determined in (\ref{sp decomp}). Then an elementary calculation shows that (\ref{line two}) equals $8(A^{\alt})_{[3]}$, and it is a standard fact that $A_{\sp(n)} = (A^{\alt})_{[3]}$. \end{proof}

\begin{Corollary} For $\alpha = 0$, we have the following formula for the codifferential of the homogeneity two component:
\begin{align}
(\codiff K^{(2)})(\xi_r) = -\frac{\scal}{2n(n+2)}I_r + 2n \tau_r^{\sharp}, \label{codiff K}
\end{align}
\no where $\tau_r^{\sharp}$ is the endomorphism associated to the Ricci-type tensor $\tau_r$ by $g(\tau_r^{\sharp}u,v) = \tau_r(u,v)$.
\end{Corollary}

\begin{proof} Using the formulae (\ref{hom 2 a}) - (\ref{hom 2 c}), we see that $(\codiff K^{(2)})(\xi_r)$ is the sum of the following three terms:
\begin{align}
-\frac{\scal}{8n(n+2)}\{\xi_t,\eta^s\} + &\frac{\scal}{8n(n+2)} \{\xi_s,\eta^t\} \,\, \mathrm{where} \,\, (r,s,t) \sim (1,2,3) \, ; \label{1 star}\\
&\sum_{a=1}^{4n} \{T_{\xi_r}(e_a),e^a\} \, ; \label{2 star}\\
-\frac{1}{2} &\sum_{a=1}^{4n} R(\{\xi_r,e^a\},e_a). \label{3 star}
\end{align}
\no By (\ref{CommF}) in the appendix, $\{\xi_t,\eta^s\} = -\{\xi_s,\eta^t\} = 2I_r$, and thus $(\ref{1 star}) = -(\scal /2n(n+2))I_r$. By lemma \ref{codiff trace}, (\ref{2 star}) vanishes, since the endomorphism $T_{\xi_r}$ is totally trace-free, and by property (iv) of theorem \ref{Biquard connection}, we have $(T_{\xi_r})_{\sp(n)} = 0$. Finally, using the identity $\{\xi_r,e^a\} = I_r(e_a)$ (cf. (\ref{CommE}) in the appendix), we see that $(\ref{3 star}) = 2n \tau_r^{\sharp}$. \end{proof}

While we don't yet wish to completely calculate the term $\mathrm{corr}(\alpha)(\xi_r)$, we note the following form for this term, which follows directly from lemma \ref{codiff trace} and by expanding the final term using $\{\xi_r,e^a\} = I_r(e_a)$ (cf. (\ref{CommE}) in the appendix):
\begin{align*}
\mathrm{corr}(\alpha)(\xi_r) = \sum_{a=1}^3 \{\{\alpha(\xi_r),\xi_a\}-\{\alpha(\xi_a),\xi_r\},\eta^a\} + \sum_{s=0}^3 \mathrm{tr}_{I_s}(\alpha(\xi_r))I_s + 8(\alpha(\xi_r))_{\sp(n)} + 4n \alpha(\xi_r).
\end{align*}
\no In particular, we see that the $\sp(\Dbdle,g)$-component of $\mathrm{corr}(\alpha)(\xi_r)$ is given by $4(n+2)(\alpha(\xi_r))_{\sp(n)}$, since the first three terms in the expression can't contribute to this component. Since, also, $(\diff \beta)_{\sp(n)} = 0$ for any section $\beta \in \Gamma(\Vbdle^*)$, it therefore follows that we must have $4(n+2)(\alpha(\xi_r))_{\sp(n)} = -(\codiff K^{(2)}\xi_r)_{\sp(n)} = -2n(\tau_r^{\sharp})_{\sp(n)}$.\\

Now, from the identity (\ref{IMV3}) which was cited above and proved in \cite{IMV07}, we see that: $$\tau_r^{\sharp} = -\frac{n+2}{2n}(I_r \circ (T^0)^{\sharp} + (T^0)^{\sharp} \circ I_r) -\frac{\scal}{8n(n+2)}I_r,$$ and the $\sp(\Dbdle,g)$-component is given by the first term. In particular, it is totally trace-free as a result of the properties of $T^0$. Using this information, we compute the term $\mathrm{corr}(\alpha)(\xi_r)$ under the following simplifying assumptions on $\alpha$: Take $\alpha(\xi_r) = (1/4)(I_r \circ (T^0)^{\sharp} + (T^0)^{\sharp} \circ I_r) + f I_r$, for $f \in C^{\infty}(M)$ (independent of the index $r$). Then $\mathrm{corr}(\alpha)(\xi_r) = \mathrm{corr}(\alpha)(\xi_r)_{\sp(1)} + 4(n+2)(\alpha(\xi_r))_{\sp(n)}$ and:
\begin{align*}
\mathrm{corr}(\alpha)(\xi_r)_{\sp(1)} &= f\sum_{a=1}^3 \{\{I_r,\xi_a\}-\{I_a,\xi_r\},\eta^a\} + \sum_{s=0}^3 \mathrm{tr}_{I_s}(fI_r) I_s + 4n fI_r \\
&= f(\{\{I_r,\xi_s\}-\{I_s,\xi_r\},\eta^s\} + \{\{I_r,\xi_t\}-\{I_t,\xi_r\},\eta^t\}) +8nfI_r\\
&= f(\{4\xi_t,\eta^s\} - \{4\xi_s,\eta^t\} + 8nfI_r \\
&= 16fI_r + 8nfI_r
\end{align*}

Now we can compute the Weyl connection for qc manifolds:

\begin{Theorem} \label{QC Weyl connection} The Weyl connection $\nabla^{qc}$ of a qc manifold $(M,\Dbdle,\mathbb{Q},[g])$ with respect to a fixed Carnot-Carath\'eodory metric $g \in [g]$, is $\nabla^{qc} = \nabla + \alpha^{qc}$, where $\alpha^{qc} \in \Gamma(\Vbdle^* \tens \mathrm{End}_0(\Dbdle))$ is represented by
\begin{align}
\alpha^{qc}(\xi_r) &= \frac{1}{4}(I_r \circ (T^0)^{\sharp} + (T^0)^{\sharp} \circ I_r) + \frac{\scal}{32n(n+2)} I_r \label{QC Weyl correction}
\end{align}
\no with respect to a choice of local qc contact form $\eta$ with corresponding local section $(I_1,I_2,I_3)$ of $\Qbdle$.
\end{Theorem}

\begin{proof} For $\alpha(\xi_r)$ of the form $fI_r + (1/4)(I_r \circ (T^0)^{\sharp} + (T^0)^{\sharp} \circ I_r)$, we have from the formula (\ref{codiff K}) and the above calculation:
\begin{align*}
(\codiff K^{\alpha(2)})(\xi_r) &= -\frac{\scal}{2n(n+2)}I_r + 8(n+2)fI_r + 2n \tau_r^{\sharp} + 4(n+2)(\alpha(\xi_r))_{\sp(n)} \\
&= -\frac{\scal}{2n(n+2)}I_r + 8(n+2)fI_r -\frac{\scal}{4(n+2)}I_r + 0 \\
&= -\frac{\scal}{4n} + 8(n+2)fI_r,
\end{align*}

\no which vanishes if and only if $f = \scal / 32n(n+2)$. \end{proof}

\section{The Rho-tensor, Weyl curvature, and local flatness}

We have determined the non-positive components $\omega^{qc}_{\leq} = \omega^{qc}_{-2} + \omega^{qc}_{-1} + \omega^{qc}_0$ of the Weyl form $\omega^{qc} = \sigma_g^*\omega^{nc}$ associated to $g$ and its exact Weyl structure $\sigma_g$. If we denote by $K^{qc}, W^{qc}$, etc. the total curvature, Weyl curvature, etc. determined by $\omega^{qc}$, then in particular the homogeneity two component $K^{qc(2)} = K^{qc(2)}_{\leq}$ of the total curvature is already determined and we have $(\codiff K^{\alpha^{qc(2)}})_{\vert \Vbdle} = 0$. Our next step is to compute $(\codiff K^{qc(2)})_{\vert \Dbdle} = (\codiff K^{qc(2)})$, which will determine $\mathsf{P}^{qc(2)}$ via the relation (\ref{homog 2 eqn}).

\begin{Lemma} We have the following formula for the homogeneity two component $K^{qc(2)}$ of the total curvature of a qc manifold $(M,\Dbdle,\mathbb{Q},[g])$ of dimension $4n+3$ determined by a choice of $g \in [g]$:
\begin{align}
- \codiff K^{qc(2)} &= \mathrm{Ric} + 2T^0 + 6U = 2(n+2)T^0 + 4(n+4) U + \frac{\mathrm{scal}}{4n}g. \label{codiff of Kqc2}
\end{align}
\end{Lemma}

\begin{proof} The calculation of $(\codiff K^{qc(2)})(u)$ for arbitrary $u \in \Dbdle$ is carried out in standard fashion from (\ref{codiff formula}) and the algebraic commutator relations (\ref{CommC})--(\ref{CommJ}), using the formulae (\ref{hom 2 a})-(\ref{hom 2 c}) for $K^{\alpha(2)}$ with $\alpha = \alpha^{qc}$ from (\ref{QC Weyl correction}) plugged in. This gives the first equality claimed in (\ref{codiff of Kqc2}), and the second equality then follows directly from the decomposition formula for $Ric$ from \cite{IMV07}, cited in (\ref{IMV1}) above.
\end{proof}

\begin{Proposition} The homogeneity two component of the Rho-tensor $\mathsf{P}^{qc(2)} \in \Gamma(\Dbdle^* \tens \Dbdle^*)$, which corresponds to $(\omega^{qc}_{+1})_{\vert T^{-1}\Gbdle_0}$, is given by:
\begin{align}
- \mathsf{P}^{qc(2)} = \frac{1}{2}T^0 + U + \frac{\mathrm{scal}}{32n(n+2)} g = L, \label{homog 2 rho tensor}
\end{align}
\no where $L$ is the tensor defined by (4.6) of \cite{IVlocalsphere}.
\end{Proposition}

\begin{proof} As discussed in section 3, $\mathsf{P}^{qc(2)}$ is completely determined from (\ref{codiff of Kqc2}) via the relation (\ref{homog 2 eqn}). It is enough to compute the action of the Kostant Laplacian $\Box$ and its inverse on $C^1_2(\g_-,\g)$, which we know to be invertible since $H^1_2(\g_-,\g) \isom \mathrm{Ker}(\Box_{C^1_2(\g_-,\g)}) = \{0\}$. Indeed, from the discussion in section 3 we have $\mathsf{P}^{qc(2)} = \Box^{-1}(\codiff K^{qc(2)}) \in \Gamma(\Dbdle^*\tens\Dbdle)$, since $\beta^{qc} := 0$ and $\alpha^{qc}$ as in (\ref{QC Weyl correction}) solve equation (\ref{homog 2 eqn simp}). In general, $\Box$ acts by scalar multiplication on the $G_0$-irreducible components of $C^1_2(\g_-,\g)$. In particular, we are concerned with the restriction to the symmetric tensors $S^2 \subset \Dbdle^* \tens \Dbdle^*$, where we have the decomposition into $Sp(1)Sp(n)$-modules $S^2 = (S^2_0)_{[-1]} \dsum (S^2_0)_{[3]} \dsum \R g$. Here one can compute directly from the formulae for $\diff$ and $\codiff$:

\begin{Lemma} On the $Sp(1)Sp(n)$-submodules $(S^2_0)_{[-1]}, (S^2_0)_{[3]}, \R g \subset \Dbdle^* \tens \Dbdle^*$, the action of the Kostant Laplacian $\Box = \diff \codiff + \codiff \diff$ is given by: $\Box_{\vert (S^2_0)_{[-1]}} = 4(n+2)\mathrm{Id}$; $\Box_{\vert (S^2_0)_{[3]}} = 4(n+4)\mathrm{Id}$; and $\Box_{\vert \R g} = 8(n+2)\mathrm{Id}$.
\end{Lemma}

\no Now we apply this directly to the $Sp(1)Sp(n)$-invariant decomposition of $\codiff K^{qc(2)}$ given by the right-hand side of the identity (\ref{codiff of Kqc2}), to get (\ref{homog 2 rho tensor}).
\end{proof}

We can immediately compute the homogeneity two component of the Weyl curvature $W^{qc}$ determined by the Weyl form $\omega^{qc} = \sigma_g^*\omega^{nc}$, giving a geometric-tensorial formulation of the sharp obstruction to local flatness of a qc structure:

\begin{Theorem} \label{Wqc curvature tensor} The homogeneity two component $W^{qc(2)}$ of the Weyl curvature of $\omega^{qc}$ is given by the following formula, viewed as a section of $\Lambda^2(\Dbdle^*) \tens \Lambda^2(\Dbdle^*)$:
\begin{align*}
W^{qc(2)}(u,v,w,z) = R(u,v,w,z) &+ (g \star L)(u,v,w,z) + \sum_{a=1}^3(\omega_a \star L_a)(u,v,w,z) \\
&+ \sum_{a=1}^3 (L_a(u,v) - L_a(v,u))\omega_a(w,z) \\
- \frac{1}{2} \sum_{a=1}^3 \omega_a(u,v)(T^0(w,I_a z) &- T^0(I_a w,z)) + \frac{\mathrm{scal}}{16n(n+2)} \sum_{a=1}^3 \omega_a(u,v)\omega_a(w,z).
\end{align*}
\no This tensor is conformally covariant and it vanishes identically if and only if the qc manifold $(M,\Dbdle,\mathbb{Q},[g])$ is locally qc isomorphic to the sphere $S^{4n+3}$ with its standard qc structure. It agrees with the tensor $WR$ derived in \cite{IVlocalsphere}.
\end{Theorem}

\begin{proof} Note that in the formula claimed, our notation is consistent with that used in \cite{IVlocalsphere}: $\star$ denotes the Kulkarni-Nomizu product of two tensors $A,B \in \Gamma(\Dbdle^*\tens\Dbdle^*)$: $$(A \star B)(u,v,w,z) := A(u,w)B(v,z) + A(v,z)B(u,w) - A(v,w)B(u,z) - A(u,z)B(v,w),$$ and for any $a=1,2,3$, we define $A_a$ by $A_a(u,v) := g(I_a \circ A^{\sharp}(u),v) = -A(u,I_a(v))$. Thus, we see that $W^{qc(2)}$ agrees with the tensor $WR$ derived in \cite{IVlocalsphere}, by comparing the above formula with display (4.8) of that text, and using the following identity which may be shown by a direct calculation for any $(r,s,t) \sim (1,2,3)$: $$L(w,I_rz) - L(I_rw,z) + L(I_sw,I_tz) - L(I_tw,I_sz) = T^0(w,I_rz) - T^0(I_rw,z).$$\\

\no The formula given here, on the other hand, can be verified directly with a calculation from the identity $$W^{qc(2)}(u,v) = K^{qc(2)}(u,v) + \{\mathsf{P}^{qc(2)}(u),v\} - \{\mathsf{P}^{qc(2)}(v),u\},$$ contained in theorem \ref{Weyl curvature} cited in section 2.1.\\

\no Now let us explain how the properties of the tensor $W^{qc(2)}$ claimed in the theorem follow from the general parabolic theory. First, since we have identified $W^{qc(2)}$ as the homogeneity two component of the Weyl curvature of $\omega^{qc} = \sigma_g^*\omega^{nc}$, it follows that $W^{qc(2)}$ corresponds to the pull-back via $\sigma_g: \Gbdle_0 \rightarrow \Gbdle$ of the homogeneity two component of the curvature $\Omega^{nc} \in \Omega^2(\Gbdle,\g)$ of the canonical Cartan connection $\omega^{nc}$ (cf. the argument in the proof of theorem 4.4 of \cite{CS03} for this fact). From the discussion in section 2.2, we see that this component of $\Omega^{nc}$ vanishes if and only if all of $\Omega^{nc}$ vanishes.\\

\no As for conformal covariance of $W^{qc(2)}$, considering it as a $(1,3)$-tensor (i.e. $W^{qc(2)}(u,v) \in \mathrm{End}_0(\Dbdle)$ for all $u,v \in \Dbdle$), a general formula for the transformation of the components $W^{qc}_i \in \Omega^2(M;\Abdle_i)$ under a change of Weyl structure $\sigma \mapsto \hat{\sigma} = \sigma \circ R_{\mathrm{exp} \circ \Upsilon_1 \ldots \mathrm{exp} \circ \Upsilon_k}$ for $\Upsilon \in \Gamma(\Abdle_+)$ is shown in section 4.6 of \cite{CS03}. (Note that any two Weyl structures $\sigma$ and $\widehat{\sigma}$ are related in this way, cf. proposition 3.2 of \cite{CS03}.) Translating that formula directly to our situation (where $W^{qc(2)}(u,v) = W^{qc}_0(u,v)$ for $u,v \in \Dbdle$), we get $$\widehat{W^{qc}_0}(u,v) = \sum_{\parallel j \parallel + l= 0} \frac{(-1)^{\underline{j}}}{\underline{j}!} \mathrm{ad}(\Upsilon_2)^{j_2} \circ \mathrm{ad}(\Upsilon_1)^{j_1}(W^{qc}_l(u,v)).$$ Since $W^{qc}_i(u,v) = 0$ for all $i < 0$, the right-hand side of this expression simplifies to $W^{qc}_0(u,v)$. In particular we have $\widehat{W^{qc(2)}} = W^{qc(2)}$ for $\widehat{W^{qc(2)}}$ the homogeneity two component of the Weyl curvature corresponding to a conformal change of the Carnot-Carath\'eodory metric to $\widehat{g} = e^{2\phi}g \in [g]$ (i.e. $\sigma_{g} \mapsto \widehat{\sigma_g} := \sigma_{\widehat{g}}$), and for the corresponding $(0,4)$-tensor given in the theorem we have $\widehat{W^{qc}(2)}(u,v,w,z) = e^{2\phi}W^{qc(2)}(u,v,w,z)$.
\end{proof}

\begin{appendix}

\section{Algebraic commutators for qc manifolds}

Let a qc manifold $(M,\Dbdle,\mathbb{Q},[g])$ of dimension $4n+3$ (assumed integrable in case $n=1$) be given, and a metric $g \in [g]$ as well as a choice of local qc contact form $(\eta^1,\eta^2,\eta^3)$ with Reeb vector fields $\xi_1,\xi_2,\xi_3$ and corresponding local section $(I_1,I_2,I_3)$ of $\mathbb{Q}$. A local section of $\Gbdle_0$ will be given by a (fixed) $g$-orthonormal local basis $\{e_1,e_2,\ldots,e_{4n}\}$, satisfying, for any $\alpha =1,\ldots,n$: $$e_{4\alpha-2} = I_1(e_{4\alpha-3}), \,\, e_{4\alpha-1} = I_2(e_{4\alpha-3}), \,\, e_{4\alpha} = I_3(e_{4\alpha-3}).$$\\

We discuss in sections 2.2 and 3 how such a local section of $\Gbdle_0$ determines pointwise bijections $\Dbdle \leftrightarrow \H^n \isom \g_{-1}$ and $\Vbdle \leftrightarrow \mathrm{Im}(\H) \isom \g_{-2}$. Namely, for $u \in \Dbdle$, write $u = \sum_{a=1}^{4n} u^a e_a$ with respect to the quaternionic unitary basis. We identify $u \leftrightarrow \overline{x} \in \H^n$, with $x = \sum_{\alpha =1}^n x_{\alpha} d_{\alpha}$, for $\{d_1,\ldots,d_n\}$ the standard quaternionic basis of $\H^n$ (considered as column vectors), and $x_{\alpha} := u^{4\alpha-3} + iu^{4\alpha-2} + ju^{4\alpha-1} + ku^{4\alpha} \in \H$. Now, for $\varphi \in \Dbdle^*$, we write $\varphi = \sum_{a=1}^{4n} \varphi_a e^a$, for $\{e^1,\ldots,e^{4n}\}$ the dual basis of $\Dbdle^*$, and get a corresponding pointwise identification by $\Dbdle^* \ni \varphi \leftrightarrow z \in (\H^n)^* \isom \g_{+1}$, for $z = \sum_{\alpha=1}^n z_{\alpha}d^{\alpha}$, where $\{d^{1}, \ldots, d^n\}$ is the dual quaternionic basis of $(\H^n)^*$ ($d^{\alpha} = (d_{\alpha})^t$) and $z_{\alpha} := \varphi_{4\alpha-3} + i\varphi_{4\alpha-2} + j\varphi_{4\alpha-1} +k\varphi_{4\alpha}$.\\

Denoting one-half the real trace form by $B$, note that we have: $$B(\overline{x},z) = B(z,\overline{x}) = \mathrm{Re}(\sum_{\alpha=1}^n z_{\alpha} \cdot \overline{x_{\alpha}}) = \varphi(u).$$ (The identification $\Dbdle^* \leftrightarrow (\H^n)^*$ is chosen precisely so that the natural dual pairing of $\Dbdle$ and $\Dbdle^*$ is compatible with the $B$-dual pairing of $\g_{-1}$ with $\g_{+1}$.) Furthermore, we have $I_s(u) \leftrightarrow \overline{i_s x} = \overline{x} \overline{i_s} = - \overline{x} i_s$ for $s=1,2,3$ and $i_1,i_2,i_3 \in \mathrm{Im}(\H)$ denoting $i,j,k$, respectively (we'll also use the notation $i_0 = 1$ as $I_0 = Id_{\Dbdle}$). So we see also: $$\sum_{\alpha=1}^n z_{\alpha}\overline{x_{\alpha}} = \varphi(u) + i \varphi(I_1(u)) + j \varphi(I_2(u)) + k \varphi(I_3(u)).$$\\

Given an endomorphism of the form $\sum_{s=0}^3 q_s I_s \in \mathrm{End}_0(\Dbdle)$, we identify it with $-\overline{q} \in \H = \R \dsum \sp(1) \hookrightarrow \g_0$, where the quaternion $q := \sum_{s=0}^3 i_s q_s$. For an endomorphism $\Phi_0 \in \sp(\Dbdle,g)$, we identify $\Phi_0$ with $A \in \gl(n,\H)$, where $A = \sum_{\alpha,\beta=1}^n \overline{A_{\alpha \beta}} E_{\alpha \beta}$, for $E_{\alpha \beta} \in \gl(n,\H)$ the matrix which sends $d_{\alpha}$ to $d_{\beta}$ and annihilates all other basis vectors, and
\begin{align*}
A_{\alpha \beta} := g(\Phi_0(e_{4\alpha-3}),&e_{4\beta-3}) + ig(\Phi_0(e_{4\alpha-3}),e_{4\beta-2}) \\
&+ jg(\Phi_0(e_{4\alpha-3}),e_{4\beta-1}) + kg(\Phi_0(e_{4\alpha-3}),e_{4\beta}).
\end{align*}
\no Then it is straightforward to calculate that the Lie bracket in $\g$ of the matrix corresponding to an endomorphism, with the matrix corresponding to a vector in $\Dbdle$, corresponds to the image of the original vector under the endomorphism. I.e. for $\mathrm{End}_0(\Dbdle) \ni \Phi \leftrightarrow A \in \g_0$ and $\Dbdle \ni u \leftrightarrow \overline{x} \in \g_{-1}$, we have $\Dbdle \ni \Phi(u) \leftrightarrow  [A,\overline{x}] \in \g_{-1}$, which means we have: $\{\Phi,u\} = \Phi(u)$.\\

For $u, v \in \Dbdle$ with $u \leftrightarrow \overline{x}$ and $v \leftrightarrow \overline{y}$, we calculate:
\begin{align*}
[\overline{x},\overline{y}] &= 2\sum_{\alpha=1}^n\mathrm{Im}(y_{\alpha}\overline{x_{\alpha}}) \in \g_{-2}\\
&= 2(ig(I_1(u),v) + jg(I_2(u),v) + kg(I_3(u),v)\\
&= id\eta^1(u,v) + jd\eta^2(u,v) + kd\eta^3(u,v).
\end{align*}
\no By regularity, we have to have $$\{u,v\} = [u,v]_{\Vbdle} = -\sum_{s=1}^3d\eta^s(u,v)\xi_s,$$ which is compatible with the identification $\Vbdle \ni \xi_s \leftrightarrow \overline{i_s} \in \mathrm{Im}(\H) \isom \g_{-2}$ given in section 3. This leads to the identification $\Vbdle^* \ni \eta^s \leftrightarrow 2i_s \in (\mathrm{Im}(\H))^* \isom \g_{+2}$, which we choose in order to get the natural compatibility condition $B(\overline{i_s},2i_r) = \delta_s^r = \eta^s(\xi_r)$ for $r,s=1,2,3$.\\

From here, one calculates the commutators of matrix elements of $\g$ resulting from the identifications, to determine the identities for the algebraic bracket induced on $TM \dsum \mathrm{End}_0(\Dbdle) \dsum T^*M$. We summarize these below. Let $u, v \in \Dbdle$, $\varphi, \psi \in \Dbdle^*$, $\Phi = \sum_{s=0}^3 q_s I_s + \Phi_0 \in \mathrm{End}_0(\Dbdle)$, $\xi = \sum_{a=1}^3 a_s \xi_s \in \Vbdle$ and $\eta = \sum_{s=1}^3 b_s \eta^s \in \Vbdle^*$. Then we have:

\begin{align}
\{u,\varphi\} = \varphi(u)I_0 - \sum_{s=1}^3 \varphi(I_s(u))I_s &+ u \wedge \varphi - \sum_{s=1}^3u\alt_{I_s} \varphi \in \mathrm{End}_0(\Dbdle); \label{CommC}\\
\{\xi,\eta\} = 2\sum_{s=1}^3 a_s b_s I_0 - 2&\sum_{(1,2,3)} (a_r b_s - a_s b_r)I_t \in \mathrm{End}_0(\Dbdle); \label{CommF}\\
\{\xi,\varphi\} = \sum_{s=1}^3 &a_s I_s(\varphi^{\sharp}) \in \Dbdle; \label{CommE}\\
\{u,v\} = [u,v]_{\Vbdle} = -2\sum_{s=1}^3 g(&I_s(u),v)\xi_s = -\sum_{s=1}^3d\eta^s(u,v)\xi_s \in \Vbdle; \label{CommB}\\
\{\Phi,\xi\} = \{\sum_{s=0}^3q_sI_s,\xi\} = &2q_0\xi + 2\sum_{(1,2,3)}(q_r a_s - q_s a_r)\xi_t \in \Vbdle; \label{CommD}\\
\{\Phi,u\} = \Phi(u) \in \Dbdle; \,\, \{\varphi,\psi\} = -&\sum_{s=1}^3 \varphi(I_s(\psi^{\sharp})) \eta^s; \,\, \{v,\eta\} = 2\sum_{s=1}^3 b_s g(I_s v, .); \label{CommH}\\
\{\Phi,\varphi\} = - \varphi \circ \Phi; \,\, \{I_r,\eta^s\} &= -\{I_s,\eta^r\} = 2 \eta^t; \, \, \{I_r,\eta^r\} = 0. \label{CommJ}
\end{align}

\no In the above formulae, $\sum_{(1,2,3)}$ denotes the sum over all cyclic permutations of $(1,2,3)$. The endomorphisms $u \alt_{I_s} \varphi$ are defined, for $s=0,1,2,3$ ($I_0$ is simply omitted above) and any $v \in \Dbdle$, by:
$$(u \alt_{I_s} \varphi): v \mapsto \varphi(I_s(v))I_s(u) - g(u,I_s(v))I_s(\varphi^{\sharp}).$$

\end{appendix}

\no \begin{small}SCHOOL OF MATHEMATICS, UNIVERSITY OF THE WITWATERSRAND, P O WITS 2050, JOHANNESBURG, SOUTH AFRICA.\end{small}\\
\no E-mail: \verb"jesse.alt@wits.ac.za"

\end{document}